\documentclass{amsart}
\usepackage{comment}

\usepackage[utf8]{inputenc}
\usepackage[all]{xy}
\usepackage{xspace}

\usepackage{amsmath,amssymb}

\usepackage{tikz-cd}
\usepackage{mathrsfs}
\usepackage{adjustbox}
\usepackage{hyperref}

\theoremstyle{plain}
\newtheorem{theorem}{Theorem}[section]
\newtheorem{proposition}[theorem]{Proposition}
\newtheorem{lemma}[theorem]{Lemma}
\newtheorem{corollary}[theorem]{Corollary}

\theoremstyle{definition}
\newtheorem{definition}[theorem]{Definition}

\theoremstyle{remark}
\newtheorem{remark}[theorem]{Remark}
\newtheorem{example}[theorem]{Example}

\newcommand{\rlc}{r-locally compact\xspace}

\newcommand{\cosupp}{\mathrm{cosupp}}
\newcommand{\op}{\mathrm{op}}
\newcommand{\hugeburnside}{\mathscr{B}_b}
\newcommand{\smallburnside}{\mathscr{B}}
\newcommand{\bigburnside}{\mathscr{B}_{lf}}
\newcommand{\lfburnside}{\mathscr{B}_{bf}}
\newcommand{\gburnside}[1]{\mathscr{B}_{#1}}
\newcommand{\bfinMod}[1]{{#1}\text{-Mod}_{fb}}
\newcommand{\bMod}[1]{{#1}\text{-Mod}_{b}}
\newcommand{\bCh}[1]{\text{Ch}_{b}(#1)}
\newcommand{\cube}[1]{\underline{2}^{#1}}
\newcommand{\cubeop}[1]{\left(\underline{2}^{#1}\right)^{\op}}
\newcommand{\cubeplus}[1]{\underline{2}^{#1}_+}
\newcommand{\cero}{\vec{0}}
\newcommand{\uno}{\vec{1}}
\newcommand{\CWSpectra}{\text{CW-Sp}}
\newcommand{\Mod}[1]{#1\text{-Mod}}
\newcommand{\Chain}[1]{\text{Ch}(#1)}
\def\CW{CW}
\def\GCW{$G$-\CW}
\newcommand{\groupCW}[1]{$#1$-\CW}

\DeclareMathOperator{\CKhannular}{CKh_{\mathbb{A}_\mathfrak{q}}}
\DeclareMathOperator{\CKhannularr}{CKh_{\mathbb{A}_\mathfrak{q}}^{r}}

\DeclareMathOperator{\colim}{colim}
\DeclareMathOperator{\dual}{D}
\DeclareMathOperator{\fdual}{D_{lf}}
\DeclareMathOperator{\linear}{\cA}

\DeclareMathOperator{\Hom}{Hom}
\DeclareMathOperator{\Map}{Map}

\DeclareMathOperator{\Tot}{Tot}

\DeclareMathOperator{\Ch}{Ch}
\DeclareMathOperator{\Top}{Top}
\newcommand{\Topp}{\mathrm{Top}_*}

\DeclareMathOperator{\hocolim}{hocolim}

\DeclareMathOperator{\ob}{ob}

\DeclareMathOperator{\Id}{Id}

\definecolor{amaranth}{rgb}{0.9, 0.17, 0.31} 
\definecolor{carrotorange}{rgb}{0.80, 0.5, 0.01} 
\definecolor{citrine}{rgb}{0.89, 0.82, 0.04} 
\definecolor{dartmouthgreen}{rgb}{0.05, 0.5, 0.06} 
\definecolor{ballblue}{rgb}{0.13, 0.67, 0.8} 
\definecolor{ceruleanblue}{rgb}{0.16, 0.32, 0.75} 
\definecolor{amethyst}{rgb}{0.6, 0.4, 0.8} 
\definecolor{amber}{rgb}{1.0, 0.75, 0.0} 
\definecolor{burlywood}{rgb}{0.87, 0.72, 0.53} 

\newcommand\lra{\longrightarrow}

\newcommand\la{\leftarrow}

\newcommand{\bE}{\mathbb{E}}

\newcommand{\bZ}{\mathbb{Z}}

\newcommand{\cA}{\mathcal{A}}

\newcommand{\cC}{\mathcal{C}}
\newcommand{\cD}{\mathcal{D}}

\newcommand{\cF}{\mathcal{F}}

\newcommand{\id}{\mathrm{id}}

\title{Quantum annular homology and bigger Burnside categories}

\author[F. Cantero-Morán]{Federico Cantero-Morán}
\address{Universidad Autónoma de Madrid and ICMAT, Spain}
\email{federico.cantero@uam.es}

\author[S.\ García-Rodrigo]{Sergio García-Rodrigo}
\address{Universidad Autónoma de Madrid, Spain}
\email{sergio.garciar@uam.es}

\author[M.\ Silvero]{Marithania Silvero}
\address{Universidad de Sevilla and IMUS, Spain}
\email{marithania@us.es}

\begin{document}

\maketitle
\begin{abstract}
As part of their construction of the Khovanov spectrum, Lawson, Lipshitz and Sarkar assigned to each cube in the Burnside category of finite sets and finite correspondences, a finite cellular spectrum. In this paper we extend this assignment to cubes in Burnside categories of infinite sets. This is later applied to the work of Akhmechet, Krushkal and Willis on the quantum annular Khovanov spectrum with an action of a finite cyclic group: we obtain a quantum annular Khovanov spectrum with an action of the infinite cyclic group.
\end{abstract}



\section{Introduction}

In 2000, Mikhail Khovanov introduced the first homological link invariant as a categorification of Jones polynomial \cite{K00}. In its construction, entirely combinatorial, one associates a bigraded chain complex to a given link diagram, such that its homology is a link invariant known as Khovanov homology.


Some years later, Lipshitz and Sarkar introduced a stable homotopy refinement of Khovanov homology~\cite{LS14}. More precisely, they provided a method to associate to a given link diagram $D$ a finite suspension spectrum $\mathcal{X}(D)$, whose cohomology is the Khovanov homology of the associated link. Together with Lawson \cite{LLS17}, they gave an equivalent construction in which they first associate to each link diagram $D$ a functor $F(D)\colon\underline{2}^n\rightarrow\mathscr{B}$ from the cube category to the Burnside category (a {\it{cube}} in the Burnside category), and then define a realization functor from the category of cubes in the Burnside category to the homotopy category of \CW-spectra. The spectrum $\mathcal{X}(D)$ is obtained as the {\it{realization}} of the cube $F(D)$.

The goal of the present paper is to extend the realization functor to cubes in several enlargements of the Burnside $2$-category: the locally finite Burnside $2$-category, the bilocally finite Burnside $2$-category and the Burnside $2$-category with an action of a group $G$ (see Section \ref{ss:burnside} for definitions). This is achieved in Sections \ref{ss:3.2}, \ref{ss:3.3} and \ref{ss:3.4}, where we respectively prove the following theorems:

\begin{theorem}\label{thm:intro1} Every cube in the locally finite Burnside category yields a \CW-spectrum, well-defined up to homotopy equivalence. Moreover, any natural transformation between such cubes induces a map between \CW-spectra well-defined up to homotopy. 
\end{theorem}

\begin{theorem}\label{thm:intro2}
Every cube in the bilocally finite Burnside category yields a locally compact \CW-spectrum, well-defined up to proper homotopy equivalence. Moreover, any natural transformation between such cubes induces a map between \CW-spectra well-defined up to proper homotopy.
\end{theorem}

A locally compact \CW-spectrum has two associated cohomologies: its usual cohomology and the compactly supported cohomology, 

\begin{theorem}\label{thm:G}
Every cube in the $G$-equivariant Burnside category yields a free \GCW-spectrum, well-defined up to $G$-homotopy equivalence. Moreover, any natural transformation between such cubes induces a $G$-equivariant map between \CW-spectra well-defined up to $G$-homotopy .
\end{theorem}

A \GCW-spectrum $\bE$ is $G$-finite if its quotient is a finite spectrum. In that case $\bE$ is also a locally compact \CW-spectrum, and thus has an associated compactly supported cohomology.

\subsection*{Application}
In \cite{APS04} Asaeda, Przytycki and Sikora introduced annular Khovanov homology, a triply-graded homological invariant of annular links (i.e., links in the thickened annulus $\mathbb{A}\times I$) which categorifies the annular Jones polynomial from \cite{HP89}. Some years later, Beliakova, Putyra and Wehrli defined a deformation of this homology theory, known as quantum annular Khovanov homology \cite{BPW19}. This arises as the homology of a chain complex $\CKhannular(D)$ with two extra gradings and with a free action of the infinite cyclic group $\bZ$. In particular, it is not finitely generated. Quotienting by the subgroup $r\bZ\subset \bZ$ one obtains a finitely generated chain complex $\CKhannularr(D)$ with an action of the finite cyclic group $\bZ_r$ of order $r$. 

Recently, Akhmechet, Krushkal and Willis \cite{AKW21} associated to each link diagram, a cube in the Burnside category with a free action of $\bZ_r$ whose realization is a finite \CW-spectrum $\mathcal{X}^r(D)$ endowed with a $\bZ_r$-action, whose cellular chain complex is $\CKhannularr(D)$.

An application of our Theorem \ref{thm:G} allows to extend their construction as follows:


\begin{theorem}\label{thm:introkhov}
Every link diagram in the annulus yields a cube in the $\bZ$-equivariant Burnside category, whose realization is a $\bZ$-finite free \groupCW{\bZ}-spectrum well-defined up to stable equivariant homotopy equivalence. This equivariant stable homotopy type does not depend on the chosen diagram. Moreover, the compactly supported cohomology of this locally compact spectrum recovers the quantum annular Khovanov homology of the link.
\end{theorem} 

The structure of the paper is as follows. In Section 2 we introduce the categories we deal with in the paper; in particular, we introduce Burnside category, its generalizations and the cubes in these categories. Section 3 is devoted to extend the realization of Lawson, Lipshitz and Sarkar to the setting of (bi)locally finite and $G$-equivariant Burnside categories and to provide the proofs of Theorems \ref{thm:intro1}--\ref{thm:G}. Section 4 compares the constructions in the previous section to the constructions of cubes in abelian groups used to define the various Khovanov homologies. Finally, in Section 5 we apply our results to the setting of quantum annular Khovanov homology and provide a proof of Theorem \ref{thm:introkhov}. The Appendix summarizes some definitions and results about locally compact spectra.

\subsection*{Acknowledgements} 
All authors are partially supported by grant SI3/PJI/2021-00505 from Comunidad de Madrid.
F. C. was funded by PID2019-108936GB-C21 from the Spanish government. M. S. was partially supported by Spanish Research Project PID2020-117971GB-C21 and by IJC2019-040519-I, funded by MCIN/AEI/10.13039/501100011033.

\section{The cube and the Burnside category}

\subsection{The cube}
The {\it{($n$-dimensional) cube}} $\cube{n}$ is the partially ordered set whose elements are $n$-tuples $u=(u_1,\ldots, u_n)\in \{0,1\}^n$ and $u\geq v$ if $u_i\geq v_i$ for every $i$. The elements of this poset are called \emph{vertices}. It has an initial element $\uno = (1,1,\ldots,1)$ and a terminal element $\cero = (0,0,\ldots,0)$. The \emph{degree} of a vertex $u$ is $\vert u\vert=\sum_{i=1}^n u_i$.

We may regard the poset $\cube{n}$ as a category whose objects are the vertices of the poset and the morphism set $\Hom(u,v)$ has a single element $\varphi_{u,v}$ if $u\geq v$ and is empty otherwise. We say that the morphism $\varphi_{u,v}$ is an \emph{edge} between vertices $u$ and $v$.

\subsection{The Burnside categories}\label{ss:burnside}

Let $G$ be a group. A {\it{$G$-set}} is a set with an action of $G$. A $G$-set is \emph{$G$-finite} if it has finitely many orbits. 

A \emph{correspondence} from a $G$-set $X$ to a $G$-set $Y$ is a triple $(A,s,t)$, where $A$ is a $G$-set and $s\colon A\rightarrow X$ and $t\colon A\rightarrow Y$ are equivariant maps, called the \emph{source} and \emph{target} map, respectively. A correspondence is
\begin{itemize}
    \item \emph{free} if the action of $G$ is free on $X,Y$ and $A$;
    \item \emph{finite} if the sets $X,Y$ and $A$ are finite;
    \item \emph{$G$-finite} if the sets $X,Y$ and $A$ are $G$-finite; 
    \item \emph{locally finite} if the source map $s$ has finite preimages (i.e., $|s^{-1}(x)| < \infty$, for every $x \in X$);
    \item \emph{bilocally finite} if both the source and target maps have finite preimages (i.e., $|s^{-1}(x)| < \infty$ and $|t^{-1}(y)| < \infty$, for every $x \in X$, $y \in Y$).
\end{itemize}

The \emph{category of correspondences} from $X$ to $Y$ has as objects the correspondences from $X$ to $Y$ and the following morphisms: a morphism from a correspondence $(A,s_A,t_A)$ to a correspondence $(B,s_B,t_B)$ is an equivariant fibrewise bijection from $A$ to $B$, that is, an equivariant bijection $f\colon A\to B$ making the following diagram commute: 
$$
{\begin{tikzcd}
    & A \arrow[rd, "t_A"] \arrow[ld, "s_A"'] \arrow[dd,"f"]   &   \\
X &                                          & Y.\\
    & B \arrow[ru, "t_B"] \arrow[lu, "s_B"']   &   
\end{tikzcd}}
$$
These bijections are composed as usual, and the identity bijection plays the role of the identity morphism.

The \emph{horizontal composition} $X\xleftarrow{s}C\xrightarrow{t}Z$ of two correspondences
$$
{\begin{tikzcd}
    & A \arrow[rd, "t_A"] \arrow[ld, "s_A"']   &&&  B \arrow[rd, "t_B"] \arrow[ld, "s_B"']   &   
\\
    X &                                          & Y &Y&& Z
\end{tikzcd}}
$$
is the fiber product or pullback $C=B\times_Y A=\{(b,a)\in B\times A \,|\, t_A(a)=s_B(b)\}$ with source and target maps given by $s(b,a)=s_A(a)$ and $t(b,a)=t_B(b)$,
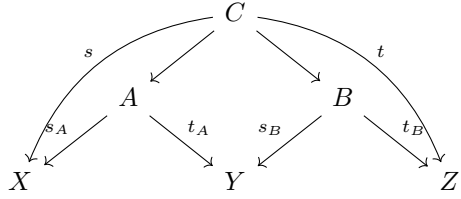
$$\begin{tikzcd}
    &                                            & C \arrow[rd] \arrow[ld] \arrow[lldd, "s"', bend right] \arrow[rrdd, "t", bend left]   &                                            &   \\
    & A \arrow[rd, "t_A"] \arrow[ld, "s_A"']   &                                                                      & B \arrow[ld, "s_B"'] \arrow[rd, "t_B"]   &   \\
X   &                                            & Y                                                                    &                                            & Z
\end{tikzcd}$$
and the action on $C$ is inherited from the diagonal action, that is, $g(b,a)=(gb,ga)$, for every $g\in G$.

The horizontal composition of two fibrewise bijections is defined as their fibrewise product too. The \emph{identity correspondence} of a $G$-set $X$ is the correspondence $X\xleftarrow{\Id} X\xrightarrow{\Id} X$.

\begin{definition}\label{def:gburnsidecat}
    The \emph{$G$-equivariant Burnside category} $\gburnside{G}$ is the $2$-category whose objects are free $G$-finite $G$-sets and whose category of morphisms from a $G$-set $X$ to a $G$-set $Y$ is the category of $G$-finite correspondences from $X$ to $Y$. The horizontal composition in $\gburnside{G}$ is the horizontal composition of correspondences and fibrewise bijections. Therefore, $1$-morphisms are correspondences and $2$-morphisms are fibrewise bijections.
\end{definition}
For the next definition, we take $G$ to be the trivial group.
\begin{definition}\label{def:burnsidecat}
The \emph{big Burnside category} $\hugeburnside$ is the 2-category whose objects are sets and whose category of morphisms from a set $X$ to a set $Y$ is the category of correspondences from $X$ to $Y$. As in the previous definition, $2$-morphisms are given by fibrewise bijections. We will consider the following subcategories:
\begin{itemize}
    \item The \emph{locally finite Burnside category} is the subcategory $\bigburnside$ of $\hugeburnside$ whose correspondences are required to be locally finite.
    \item The \emph{bilocally finite Burnside category} is the subcategory $\lfburnside$ of $\hugeburnside$ whose correspondences are required to be bilocally finite.
    \item The \emph{Burnside category} is the subcategory $\smallburnside$ of $\hugeburnside$ whose correspondences are required to be finite.
\end{itemize}
\end{definition}

\begin{lemma}\label{lemma:forgetfulfactors}
	The forgetful functor $\gburnside{G}\to \hugeburnside$ factors through the bilocally finite Burnside category $\lfburnside$.
\end{lemma}
\begin{proof}
If $f\colon A\to B$ is an equivariant map to a free $G$-set $B$ and $b\in B$, then $f^{-1}(b)$ contains at most one element from each orbit. If there are finitely many orbits, then $f^{-1}(b)$ is finite. Applying this to the source and target maps in a $G$-finite correspondence $X\xleftarrow{s} A\overset{t}{\to} Y$ yields the result.
\end{proof}

Let $R$ be a commutative ring. The \emph{linearization functor} 
\begin{align*}
    \linear\colon\bigburnside&\lra \Mod{R}
\end{align*}
maps a $G$-set $X$ to the free $R$-module $\mathcal{A}(X)=R\langle X\rangle$ generated by $X$ and a correspondence $X\xleftarrow{s}A\xrightarrow{t}Y$ to the homomorphism 
\begin{align*}
    \linear(A,s,t)\colon\linear(X)  \longrightarrow & \linear(Y),
\end{align*}
whose value on a generator $x\in X$ is:
\[
    \sum_{y\in Y}\vert s^{-1}(x)\cap t^{-1}(y)\vert \cdot  y=\sum_{a\in s^{-1}(x)}t(a).
\]
Precomposing this functor with the forgetful functor $\gburnside{G}\to \lfburnside\subset \bigburnside$, we obtain a functor 
\[
    \linear\colon{\gburnside{G}}\lra \Mod{R[G]},
\]
since the action of $G$ makes $\linear(X)$ a $R[G]$-module.

\subsection{Cubes in the Burnside category and their totalizations} In this section we extend \cite[Definitions 4.2, 4.3 and Lemma 4.4]{LLS20} to the setting of the big Burnside category. 

\begin{definition}
A \emph{cube in the big Burnside category} is a (strictly unitary lax) $2$-functor $F$ from the category $\cube{n}$ to the $2$-category $\hugeburnside$. In detail, it consists of the following data:
\begin{itemize}
\item For each $u\in \{0,1\}^n$, a set $F(u)$.
\item For every $u> v$, a correspondence $F(u)\xleftarrow{s}F(\varphi_{u,v})\xrightarrow{t}F(v)$.
\item For every $u> v> w$, a $2$-morphism $F_{u,v,w}\colon F(\varphi_{v,w})\times_{F(v)}F(\varphi_{u,v})\rightarrow F(\varphi_{u,w})$ such that the following diagram commutes for every $u> v> w> z$:
$$\begin{tikzcd}
F(\varphi_{w,z})\times_{F(w)} F(\varphi_{v,w})\times_{F(v)} F(\varphi_{u,v}) \arrow[r,"\Id\times F_{u,v,w}"]\arrow[d,"F_{v,w,z}\times \Id"'] & F(\varphi_{w,z})\times_{F(w)} F(\varphi_{u,w}) \arrow[d,"F_{u,w,z}"] \\
F(\varphi_{v,z})\times_{F(v)} F(\varphi_{u,v}) \arrow[r,"F_{u,v,z}"'] & F(\varphi_{u,z})
\end{tikzcd}$$
\end{itemize}
\end{definition}

We may replace the big Burnside category by any of the subcategories introduced in Definition \ref{def:burnsidecat}, or by the $G$-equivariant Burnside category $\gburnside{G}$. Let $\hugeburnside^{\cube{n}}$ be the functor category, whose objects are strictly unitary lax $2$-functors from $\cube{n}$ to $\hugeburnside$ and morphisms are natural transformations between these functors. 

Let $F\colon \cube{n}\to \bigburnside$ be a strictly unitary lax $2$-functor.
\begin{definition}
The \emph{totalization of $F$} is the chain complex of $R$-modules with underlying graded module
$$\Tot(F)=\bigoplus_{u\in\underline{2}^n}\linear(F(u)),$$ with homological degree of $\linear(F(u))$ equal to $\vert u\vert$ and differential 
\[ \partial\colon\bigoplus_{\vert u\vert=i}\linear(F(u))\rightarrow\bigoplus_{\vert v\vert=i-1}\linear(F(v))
\] 
sending a generator $x\in F(u)$ to the sum $\partial(x) = (-1)^{s_{u,v}}\linear(\varphi_{u,v})(x)$ if $\vert u \vert = \vert v\vert+1$ and to $0$ otherwise. Define $s_{u,v} = \sum_{j<k} u_j$ where $k$ is the only coordinate such that $u_k\neq v_k$.
\end{definition}

This chain complex can alternatively be constructed as the following homotopy colimit: Define the \emph{extended cube category of dimension $n$}, $\cubeplus{n}$, as the result of adding to the cube category an extra object $*$ that receives an arrow $\varphi_{u,*}$ from every vertex $u$ in $\cube{n} \setminus \cero$. Then the cube of $R$-modules $\linear\circ F$ can be extended to a functor $K\colon \cubeplus{n}\to \Mod{R}$ by setting $F(*)$ to be the trivial module. As a consequence we have \cite[Section 2.1]{LS17}:
\begin{lemma}
    The totalization of $F$ is homotopy equivalent to the homotopy colimit of $K$.
\end{lemma}

This rule is functorial and defines the \emph{totalization functor} to the category of chain complexes of $R$-modules
\[
    \Tot\colon\bigburnside^{\cube{n}}\rightarrow\Chain{R}.
\] 

\section{Cubes in Burnside categories and spectra}
In \cite{LLS20} Lawson, Lipshitz and Sarkar defined a functor from the category of cubes in the Burnside category to the homotopy category of \CW-spectra. In this section we explain how to promote this construction to cubes in the locally finite Burnside category. Then, we show that for cubes in the bilocally finite Burnside category, the \CW-spectra obtained are locally compact (see the Appendix for the definition of locally compact \CW-spectrum). Finally, we explain that the realization of certain cubes in a $G$-equivariant Burnside category yield locally compact \GCW-spectra.

When defining the geometric realization of the cubes above, we need to take the homotopy colimit of some functors satisfying certain properties, called homotopy coherent diagrams; we recall these notions:

\begin{definition}\label{def:homotopycoherentdiag}
A {\it{homotopy coherent diagram}} from a small category $\mathscr{C}$ to the category of based topological spaces $\Top_\ast$, $F\colon \mathscr{C} \to \Top_\ast$, consists of:
\begin{itemize}
\item For each $u\in \ob(\mathscr{C})$ a space $F(u)\in \ob(\Top_\ast)$.
\item For each $n\geq 1$ and each sequence $u_0 \xrightarrow{f_1}u_1\xrightarrow{f_2}\dots\xrightarrow{f_n}u_n$ of composable morphisms in $\mathscr{C}$, a continuous map $$F(f_n,\dots,f_1)\colon[0,1]^{n-1}\times F(u_0)\longrightarrow F(u_n)$$ with $F(f_n,\dots,f_1)([0,1]^{n-1}\times\{\ast\})=\ast$, satisfying $F(f_n,\dots,f_1)(t_1,\dots,t_{n-1})=$
\[\left\lbrace\begin{array}{ll}
F(f_n,\dots,f_2)(t_2,\dots,t_{n-1}) & f_1=\Id; \\
F(f_n,\dots,f_{i+1},f_{i-1},\dots,f_1)(t_1,\dots,t_{i-1}t_i,\dots,t_{n-1}) & f_i=\Id, \, 1<i<n; \\
F(f_{n-1},\dots,f_1)(t_1,\dots,t_{n-2}) & f_n=\Id; \\
{[}F(f_n,\dots,f_{i+1})(t_{i+1},\dots,t_{n-1})]\circ [F(f_i,\dots,f_1)(t_1,\dots,t_{i-1})] & t_i=0; \\
F(f_n,\dots,f_{i+1}\circ f_i,\dots,f_1)(t_1,\dots,t_{i-1},t_{i+1},\dots,t_{n-1}) & t_i=1.
\end{array}\right.\]
\end{itemize}
\end{definition}
\vspace{0.1cm}

Given a homotopy coherent diagram $F\colon \mathscr{C} \to \Top_\ast$, its homotopy colimit can be computed as the following quotient 
\begin{equation}\label{hocolim}
\hocolim F\simeq\{\ast\}\amalg\coprod_{n\geq 0}\coprod_{u_0\xrightarrow{f_1}\dots\xrightarrow{f_n}u_n}[0,1]^n\times F(u_0)/\sim,
\end{equation}
where $u_0\xrightarrow{f_1}\dots\xrightarrow{f_n}u_n$ is a sequence of composable morphisms in $\mathscr{C}$. If $(t_1,\dots,t_n)\in[0,1]^n$, $p\in F(u_0)$ and $\ast$ is the basepoint, then the equivalence relation is given by $(f_n,\dots,f_1; t_1,\dots,t_n;p)\sim$ 
\[\left\lbrace\begin{array}{ll}
(f_n,\dots,f_2;t_2,\dots,t_n;p) & f_1=\Id; \\
(f_n,\dots,f_{i+1},f_{i-1},\dots,f_1;t_1,\dots,t_{i-1}t_i,\dots,t_n;p) & f_i=\Id, \, 1<i; \\
(f_n,\dots,f_{i+1};t_{i+1},\dots,t_n;F(f_{i},\dots,f_1)(t_1,\dots,t_{i-1};p)) & t_i=0; \\
(f_n,\dots,f_{i+1}\circ f_{i},\dots,f_1;t_1,\dots,t_{i-1},t_{i+1},\dots,t_n;p) & t_i=1, \, i<n; \\
(f_{n-1},\dots,f_1;t_1,\dots,t_{n-1};p) & t_n=1; \\
\ast & p=\ast.
\end{array}\right.\]
We refer to this model of the homotopy colimit as the Vogt homotopy colimit \cite{V73}.

Observe that in the particular case when all isomorphisms in the category $\mathscr{C}$ correspond to identity maps, one can assume that none of the maps $f_i$ considered when $n>0$ are equal to the identity, and take the equivalence relation generated by the last four conditions when constructing the homotopy colimit of $F$ (see \cite[Observation~4.12]{LLS20}).

\subsection{The realization of a cube in the Burnside category}\label{SectionRealizationBurnside}

Given a set $X$, denote by $X_+$ the disjoint union of $X$ together with a basepoint, and let $S^k$ be the pointed sphere of dimension $k$. Recall from \cite{LLS20} that a \emph{$k$-dimensional box map} refining a correspondence $X\overset{s}{\la} A\overset{t}{\to} Y$ is a continuous map $g\colon X_+\wedge S^k\to Y_+\wedge S^k$ satisfying:
\begin{enumerate}
    \item $g$ is the composition of a continuous map $f\colon X_+\wedge S^k\to A_+\wedge S^k$ with the product map $t\wedge \Id\colon A_+\wedge S^k\to Y_+\wedge S^k$.
    \item $f$ is the collapse map associated to a rigid embedding $e\colon [0,1]^k \times A\to [0,1]^k\times X$ such that $e(x,a)\in [0,1]^k\times \{s(a)\}$.
\end{enumerate}

The set of all box maps $E(s)$ refining a given correspondence $X\overset{s}{\la} A\overset{t}{\to} Y$ is in bijection with a subset of the space of all embeddings of $[0,1]^k \times A$ into $[0,1]^k\times X$, and this endows $E(s)$ with a topology. The inclusion of the set of box maps $E(s)$ into the set of continuous maps $X_+\wedge S^k\to Y_+\wedge S^k$ is then continuous and the space $E(s)$ is $(k-2)$-connected.

\begin{definition}\label{definition:spatref} Given a strictly unitary lax $2$-functor $F\colon \cC \to \smallburnside$ from a small category $\cC$ to $\smallburnside$, a \emph{$k$-spatial refinement} of $F$ is a homotopy coherent diagram $\tilde{F}\colon \cC\to \Topp$ consisting of: 
\begin{enumerate}
	\item For each $u\in \ob(\cC)$, the space $\tilde{F}(u) = F(u)_+\wedge S^k$.
	\item For each morphism $f\colon u\to v$ in $\cC$, a box map $\tilde{F}(f)$ refining the correspondence $F(f)$ from $F(u)$ to $F(v)$.
	\item For each sequence $u_0\overset{f_1}{\to} u_1\overset{f_2}{\to} \ldots \overset{f_n}{\to} u_n$ of composable morphisms 
 in $\cC$, a continuous family 
    \[
        \tilde{F}(f_n,\ldots,f_1)\colon [0,1]^{n-1}\to \Map(\tilde{F}(u_0),\tilde{F}(x_n))
    \]
    of box maps whose restriction to $[0,1]^{n-k-1}\times \{0\}\times [0,1]^{k-1}$ is 
    \[
        \tilde{F}(f_n,\ldots,f_{k+1})\times \tilde{F}(f_{k},\ldots,f_1)
    \]
    and whose restriction to $[0,1]^{n-k-1}\times \{1\}\times [0,1]^{k-1}$ is 
    \[
        \tilde{F}(f_n,\ldots,f_{k+1}\circ f_{k},\ldots,f_1).
    \]
\end{enumerate} 

When the families of box maps $\tilde{F}(f_n,\ldots,f_1)$ are only defined for $n\leq \ell$, we say $F$ is an \emph{$\ell$-partial $k$-spatial refinement}.
\end{definition}

\begin{proposition}\cite[Prop.~5.22]{LLS20}\label{prop:2} Consider a small category $\cC$ such that every sequence of composable nonidentity morphisms has length at most $n$, and let $F \colon \cC \to~\smallburnside$. 
    \begin{enumerate}
	   \item If $k\geq n$ then there is a $k$-spatial refinement of $F$.
	   \item If $k\geq n+1$ then any two $k$-spatial refinements of $F$ are homotopic (as homotopy coherent diagrams). 
	   \item If $\tilde{F}_k$ is a $k$-spatial refinement of $F$ then the result of suspending each $\tilde{F}_k(u)$ and $\tilde{F}_k(f_1,\ldots,f_n)(\vec{t})$ gives a $(k+1)$-spatial refinement of $F$.
    \end{enumerate}
\end{proposition}

The proof of the above proposition implicitly uses a recursive application of the following lemma.

\begin{lemma}\label{lemma:recursion} If $k\geq \ell$ and $\tilde{F}$ is an $(\ell-1)$-partial $k$-spatial refinement of $F$, then there is an $\ell$-partial $k$-spatial refinement of $F$ extending $\tilde{F}$.
\end{lemma}

We are ready now to describe the realization $|F|$ of a cube $F \colon \cube{n} \to \mathscr{B}$ following~\cite{LLS20}: Consider the extended cube category $\cubeplus{n}$ and choose a $k$-spatial refinement $\tilde{F}$ of $F$ for some $k>n$ (this can be done by Proposition \ref{prop:2}(1)). The next step is to extend $\tilde{F}$ to a homotopy coherent diagram $\tilde{F}_*$ indexed by $\cubeplus{n}$ by setting $\tilde{F}(*)$ to be the basepoint and $\tilde{F}(\varphi_{u,*})$ the constant map. If we denote by $\|\tilde{F}\|$ the homotopy colimit of the homotopy coherent diagram $\tilde{F}_*$, the geometric realization of $F$ is defined as $$|F| = \Sigma^{-k}\Sigma^\infty\|\tilde{F}\|,$$ the $k$-fold desuspension of the suspension spectrum of $\|\tilde{F}\|$. 

\begin{theorem}\cite[Section 5.3]{LLS20}
The realization $|F|$ of a cube in the Burnside category is a \CW-spectrum well-defined up to homotopy. Moreover, any natural transformation between cubes in the Burnside category induces a map between \CW-spectra well-defined up to homotopy. 
\end{theorem}

\subsection{The realization of a cube in the locally finite Burnside category}\label{ss:3.2}

Let $f\colon X\to Y$ be a correspondence given as $X\overset{s}{\leftarrow} A\overset{t}{\to} Y$ and $x\in X$. Define the \emph{cosupport of $x$ under $f$} as the subset
\[
    \cosupp(x,f) = t(s^{-1}(x))\subset Y.
\]
An inclusion $\iota\colon X'\subset X$ gives rise to a correspondence $X'\xleftarrow{\id} X' \xrightarrow{\iota} X$ that we denote by $\iota$ too. Write $f|_{X'}$ for the restriction of $f$ to $X'$, i.e., for the composition of correspondences $f\circ \iota$. 

Let $\cC$ be a small category and $F\colon \cC\to \bigburnside$ a strictly unitary lax $2$-functor. If $u_0\overset{f_1}{\to} u_1\overset{f_2}{\to} \ldots \overset{f_\ell}{\to} u_\ell$ is a sequence of composable morphisms in $\cC$ and $x\in F(u_0)$, define a functor $F_x(f_1,\ldots,f_\ell)\colon \cC\to \smallburnside$ whose value on objects is 
\[
	F_x(u) = \begin{cases}
	\{y\in F(u_i)\mid y\in \cosupp(u,F(f_i\circ\ldots\circ f_1))\} & \text{if $u = u_i$,} \\
	\emptyset & \text{if $u\neq u_i$ for all $i$,}
\end{cases}
\]
and whose value on a morphism $f\colon u\to v$ in $\cC$ is $F_x(f) = F(f)|_{F_x(u)}$.

A \emph{($\ell$-partial) $k$-spatial refinement} of a functor $F\colon \cC\to \bigburnside$ is defined as in Definition \ref{definition:spatref}, replacing the Burnside category by the locally finite Burnside category.

\begin{lemma}\label{lemma:recursiveb} Let $F\colon \cC\to \bigburnside$ be a strictly unitary lax $2$-functor. If $k\geq \ell$ and $\tilde{F}$ is an $(\ell-1)$-partial $k$-spatial refinement of $F$, then there is an $\ell$-partial $k$-spatial refinement of $F$ extending $\tilde{F}$.
\end{lemma}
\begin{proof}
Let $u_0\overset{f_1}{\to} u_1\overset{f_2}{\to} \ldots \overset{f_\ell}{\to} u_\ell$ be a sequence of $\ell$ composable morphisms in $\cC$. For each $x\in F(u_0)$, the functor $F_x(f_1,\ldots,f_\ell)$ lies in the Burnside category $\smallburnside$. By Lemma \ref{lemma:recursiveb}, we can extend $F_x(f_1,\ldots,f_\ell)$ to a spatial refinement on $f_1,\ldots,f_\ell$. We define $\tilde{F}(f_1,\ldots,f_\ell) = \bigvee_{x\in F(u_0)} \tilde{F}_x(f_1,\ldots,f_\ell)$.
\end{proof}

The following result is an extension of Proposition \ref{prop:2} to the setting of the locally finite Burnside category. It will be crucial in the realization of the cubes in this category. 

\begin{proposition}\label{prop:spref}
Consider $\cC$ a small category such that every sequence of composable nonidentity morphisms has length at most $n$, and let $F \colon \cC \to \bigburnside$. 
\begin{enumerate}
	\item If $k\geq n$ then there is a $k$-spatial refinement of $F$.
	\item If $k\geq n+1$ then any two $k$-spatial refinements of $F$ are homotopic (as homotopy coherent diagrams). 
	\item If $\tilde{F}_k$ is a $k$-spatial refinement of $F$ then the result of suspending each $\tilde{F}_k(u)$ and $\tilde{F}_k(f_1,\ldots,f_n)(\vec{t})$ gives a $(k+1)$-spatial refinement of $F$.
\end{enumerate}
\end{proposition}
\begin{proof}
On objects, we set $\tilde{F}(u) = F(u)_+\wedge S^k$. Then we extend the spatial refinement recursively on the length of the chain of composable morphisms using Lemma~\ref{lemma:recursiveb}. This completes the proof of (1). The proof of (2) and (3) is analogous to that in \cite[Proposition 5.22]{LLS20}.
\end{proof}

Given a cube in the locally finite Burnside category $F \colon \cube{n} \to \bigburnside$, define its realization $|F|$ in the same way as described in Section \ref{SectionRealizationBurnside} for the case of the Burnside category. 

\begin{proposition} Let $F\colon \cube{n}\to \bigburnside$ be a strictly unitary lax $2$-functor and $k>n$.
\begin{enumerate}
    \item If $\tilde{F}_1$ and $\tilde{F}_2$ are two $k$-spatial refinements of $F$, then $\|\tilde{F}_1\|\simeq \|\tilde{F}_2\|$. 
    \item If $\tilde{F}$ is a $k$-spatial refinement of $F$, then $\|\Sigma \tilde{F}\|\simeq \Sigma \|\tilde{F}\|$.
    \item If $\eta \colon F\to F'$ is a natural transformation between two cubes $F,F'\colon \cube{n}\to \bigburnside$, and $\tilde{F}$ and $\tilde{F'}$ are $k$-spatial refinements of $F$ and $F'$, respectively,  then $\eta$ induces a map $\|\eta\| \colon \|\tilde{F}\|\to \|\tilde{F'}\|$ well-defined up to homotopy.
\end{enumerate}	
\end{proposition}
\begin{proof}
(1) and (2) are consequences of Proposition \ref{prop:spref} (2) and (3), respectively. To prove (3), observe that a natural transformation from $F$ to $F'$ is a (strictly unitary lax) 2-functor $\underline{\eta} \colon \cube{n+1}\to \bigburnside$ that restricts to $F$ and $F'$ on $\cube{n}\times \{1\}$ and $\cube{n}\times \{0\}$, respectively. By Lemma \ref{lemma:recursiveb}, the spatial refinements $\tilde{F}$ and $\tilde{F'}$ can be extended to a spatial refinement of $\underline{\eta}$. This spatial refinement yields a map from the coherent diagram $\tilde{F}$ to the coherent diagram $\tilde{F'}$ that commutes up to homotopy, which in turn yields a map between the \CW-complexes $\|\tilde{F}\|$ and $\|\tilde{F'}\|$.
\end{proof}

\begin{corollary}
The realization $|F|$ of a cube in the locally finite Burnside category $\bigburnside$ is a \CW-spectrum well-defined up to homotopy. Moreover, any natural transformation between cubes in the locally finite Burnside category induces a well-defined map up to homotopy. 
\end{corollary}

This completes the proof of Theorem \ref{thm:intro1}.

\subsection{The realization of a cube in the bilocally finite Burnside category} \label{ss:3.3}

We will prove that the realization of cubes in the bilocally finite Burnside category leads to locally compact \CW-spectra. To do this, we first prove a result in the setting of functors over locally finite categories.

A \emph{proper homotopy coherent diagram} is a homotopy coherent diagram in the category of \rlc \CW-complexes and proper maps \CW$_{lc}$ (see Appendix for the definition of this category). A \emph{proper homotopy coherent natural transformation} $\eta\colon F\to F'$ between proper homotopy coherent diagrams $F,F'\colon \cD\to \CW_{lc}$ is a proper homotopy coherent diagram $\underline{\eta}\colon \cD\times \cube{1}\to \CW_{lc}$. A \emph{proper homotopy} $H\colon F\times [0,1]\to F'$ between two proper homotopy natural transformations $\eta,\eta'$ is a proper homotopy coherent natural transformation $\underline{H}\colon \cD\times \cube{1}\to \CW_{lc}$ between $F\times [0,1]$ and $F'$ such that the restriction to $F\times \{1\}$ is $\eta$ and the restriction to $F\times \{0\}$ is $\eta'$.

\begin{lemma}\label{lemma:propervogt} Let $\cC$ be a small category such that for each object $u$ of $\cC$ the comma categories $\cC\downarrow u$ and $u\downarrow \cC$ have finitely many sequences of composable non-identity morphisms. The Vogt homotopy colimit of a proper homotopy coherent diagram $F$ over $\cC$ is a locally compact space. If $\eta\colon F\to F'$ is a level-wise proper natural transformation of proper homotopy coherent diagrams, then $\eta$ induces a proper map between Vogt homotopy colimits. If $\eta$ and $\eta'$ are two naturally properly homotopic natural transformations, then the induced maps between Vogt homotopy colimits are properly homotopic.
\end{lemma}

\begin{proof}
The Vogt homotopy colimit of such a diagram is the quotient of a disjoint union of a basepoint $\ast$ and spaces of the form $[0,1]^m\times F(u_0)$, indexed by sequences $u_0\xrightarrow{f_1}\dots\xrightarrow{f_m}u_m$ of composable non-identity morphisms. Observe first that this is a disjoint union of compact spaces, and therefore is locally compact. 
To show that the quotient is locally compact, notice first that every equivalence class has a unique representative $\mathfrak{g}=(g_\ell,\dots,g_1;s_1,\dots,s_\ell;q)$ such that $(s_1,\dots,s_\ell)\in (0,1)^\ell$ and $q\in F(v_0)$, where $v_0\xrightarrow{g_1}\dots\xrightarrow{g_\ell}v_\ell$. Any other element in this equivalence class is of the form $\mathfrak{f}=(f_m,\dots,f_1;t_1,\dots,t_m;p)$ with $m>\ell$ and $t_{j_i}=s_i$ with $j_{1}< \ldots < j_{\ell}$ and the remaining $t_j$ are equal to either 0 or 1.

In fact, $\mathfrak{f}$ is obtained from $\mathfrak{g}$ by a sequence of the following three transformations, where $\mathfrak{h}=(h_k,\dots,h_1;r_1,\dots,r_k;z)$ (compare to relations in Definition \ref{def:homotopycoherentdiag}):

\begin{itemize}
    \item[(a)] $\mathfrak{h} \, \rightarrow \, (h_{k+1},h_k,\dots,h_1;r_1,\dots,r_k,1;z)$, where $h_{k+1}\colon u_k\rightarrow u_{k+1}$ is an object in $u_k\downarrow\cC$.
    \item[(b)] $\mathfrak{h} \, \rightarrow \, (h_k,\dots,h_1,h_0;0,r_1,\dots,r_k;y)$, where $h_0\colon u_{-1}\rightarrow u_0$ is an object in $\cC\downarrow u_0$ and $F(h_0)(y)=z$.
    \item[(c)] $\mathfrak{h} \, \rightarrow \, (h_k,\dots,h_j^{\prime\prime},h_j^\prime,\dots,h_1;r_1,\dots,1,r_j,\dots,r_k;z)$, where $h_j^\prime\colon u_{j-1}\rightarrow w$, $h_j^{\prime\prime}\colon w\rightarrow u_j$ for some $w$, and $h_j=h_j^{\prime\prime}\circ h_j^\prime$ with $1\leq j \leq k$.
\end{itemize}

The sequence transforming $\mathfrak{g}$ into $\mathfrak{f}$ can be chosen so that transformations of type (a) occur before those of type (b) and transformations of type (c) are the last in the sequence. 

Let $A$ be the set of sequences of transformations of type (a) that can be applied to $\mathfrak{g}$. For each $a\in A$ let $\mathfrak{g}_a$ be the representative defined by $a$ and let $B_a$ be the set of sequences of transformations of type (b) that can be applied to $\mathfrak{g}_a$. Finally, for each $b\in B_a$, let $\mathfrak{g}_{a,b}$ be the representative obtained by applying $b$ to $\mathfrak{g}_a$ and let $C_{a,b}$ be the set of all representatives of $\mathfrak{g}$ that are obtained by applying transformations of type (c) to $\mathfrak{g}_{a,b}$.

Since the comma category $v_\ell\downarrow\cC$ has finitely many sequences of composable non-identity morphisms, $A$ is finite. The same condition over $\cC \downarrow v_0$ implies that $B_a$ is finite. Lastly, since $u\downarrow \cC$ (or $\cC\downarrow u$) has finitely many sequences of composable non-identity morphisms, $C_{a,b}$ is finite.

Since $A, B_a, C_{a,b}$ are all finite, there are finitely many representatives in each class and therefore the quotient is locally compact. 

If $\eta\colon F\rightarrow F^\prime$ is a level-wise proper homotopy coherent natural transformation, we can view it as a proper homotopy coherent diagram $\underline{\eta}\colon\cC\times\cube{1}\rightarrow\Topp$ such that it restricts to $F$ and $F^\prime$ on $\cC\times\{1\}$ and $\cC\times\{0\}$, respectively. There is a proper homotopy equivalence $\hocolim \underline{\eta} \simeq \hocolim F^\prime$ and a proper inclusion  $\hocolim F\to \hocolim \underline{\eta}$. The composition of these two maps is the map between homotopy colimits. 

If $\eta,\eta'\colon F\to F^\prime$ are level-wise proper homotopy coherent natural transformations and $H\colon F\times [0,1]\to F^\prime$ is a proper homotopy coherent diagram that restricts to $\eta$ and $\eta'$ on $F\times \{1\}$ and $F\times \{0\}$, then its Vogt homotopy colimit $\hocolim H$ is properly homotopy equivalent to $\hocolim F^\prime$, and has an inclusion $(\hocolim F)\times [0,1]\to \hocolim H$. Their composition is a proper homotopy equivalence between the maps induced by $\eta$ and $\eta'$ on Vogt homotopy colimits.
\end{proof}

Given a cube in the bilocally finite Burnside category $F \colon \cube{n} \to \lfburnside$, consider its realization $|F|$ in the same way as described for Burnside category in Section~\ref{SectionRealizationBurnside} using the Vogt homotopy colimit \eqref{hocolim}. 

\begin{proposition}\label{lemma:lfCWspectrum}
The realization $|F|$ of a cube in the bilocally finite Burnside category $\lfburnside$ is a locally compact \CW-spectrum well-defined up to proper homotopy. Moreover, any natural transformation between cubes in the bilocally finite Burnside category induces a well-defined map up to proper homotopy.
\end{proposition}

\begin{proof}
A spatial refinement of a cube in $\lfburnside$ maps each vertex to a wedge of spheres, which is a \rlc \CW-complex. An edge is sent to the realization of the spatial refinement of a bilocally finite correspondence $f$, which is a proper map $\tilde{f}\colon X\to Y$, where $X$ and $Y$ are wedges of spheres. Finally, the homotopies that arise from isotopies of box maps are also proper: if $H\colon X\times I\to Y$ is the homotopy constructed between the spatial refinements $\tilde{f},\tilde{g}\colon X\to Y$ of two bilocally finite correspondences $f$ and $g$, and $e$ is an open cell of $Y$, then $H^{-1}(e)$ only meets product cells $e'\times I\subset X\times I$ with $e'$ a cell that meets $f^{-1}(e)$. Since there are finitely many of the latter, $H$ is proper. As a consequence, the spatial refinement built in Proposition \ref{prop:spref} is a proper homotopy coherent diagram. Then, by Lemma \ref{lemma:propervogt}, $\Vert F\Vert$ is a locally compact \CW-complex, hence locally finite, and from Remark~\ref{remark_A1} it follows that $\vert F\vert$ is a locally finite \CW-spectrum.

To prove that the proper homotopy type of $|F|$ is well-defined, suppose that we are given two spatial refinements $\tilde{F}_1$ and $\tilde{F}_2$ and let $I=\cube{1}$ be the poset $(1\rightarrow 0)$.

By Lemma \ref{lemma:recursiveb}, we can extend $\tilde{F}_1$ and $\tilde{F}_2$ to a spatial refinement $\tilde{F}_a\colon\cube{n}\times I\rightarrow\Topp$ that restricts to $\tilde{F}_1$ on $\cube{n}\times \{1\}$ and to $\tilde{F}_2$ on $\cube{n}\times \{0\}$, and is the identity on $\{u\}\times I$. Then, $\tilde{F}_a$ induces a proper map $\Vert\tilde{F}_a\Vert\colon\Vert\tilde{F}_1\Vert\rightarrow\Vert\tilde{F}_2\Vert$.

Similarly, we can produce a spatial refinement $\tilde{F}_b\colon\cube{n}\times I\rightarrow\Topp$ that restricts to $\tilde{F}_2$ on $\cube{n}\times \{1\}$ and to $\tilde{F}_1$ on $\cube{n}\times \{0\}$, and is the identity on $\{u\}\times I$. It induces a proper map $\Vert\tilde{F}_b\Vert\colon \|\tilde{F}_2\|\to \|\tilde{F}_1\|$.

The compositions $\Vert\tilde{F}_b\Vert\circ \Vert\tilde{F}_a\Vert$ and $\Vert\tilde{F}_a\Vert\circ \Vert\tilde{F}_b\Vert$ are induced by the diagrams that result from gluing $\tilde{F}_a$ and $\tilde{F}_b$, and $\tilde{F}_b$ and $\tilde{F}_a$, respectively. Using once more Lemma~\ref{lemma:recursiveb}, define another spatial refinement $\tilde{F}\colon \cube{n}\times I\times I\rightarrow\Topp$ that restricts to $\tilde{F}_a\circ \tilde{F}_b$ on $\cube{n}\times I\times \{1\}$ and to the identity on $\cube{n}\times I\times \{0\}$, with arrows $\{u\}\times\{i\}\times I$ being identities. 

Then $\tilde{F}$ induces a proper homotopy between $\Vert\tilde{F}_b\Vert\circ \Vert\tilde{F}_a\Vert$ and the identity. In a similar way we obtain a proper homotopy between $\Vert\tilde{F}_a\Vert\circ \Vert\tilde{F}_b\Vert$ and the identity, so $\Vert\tilde{F}_1\Vert$ and $\Vert\tilde{F}_2\Vert$ are proper homotopy equivalent. By Remark \ref{remark_A1}, $\vert F\vert$ is well-defined up to proper homotopy type.

Consider now a natural transformation $\eta\colon F\rightarrow F^\prime$ between cubes in the bilocally finite category $F,F^\prime\colon\cube{n}\rightarrow\lfburnside$, and let $\tilde{F}$ and $\tilde{F^\prime}$ be two spatial refinements of $F$ and $F^\prime$, respectively. Then $\eta$ induces a map between the spatial refinements $\tilde{\eta}\colon\tilde{F}\rightarrow\tilde{F^\prime}$, as well as between the \CW-complexes, $\Vert\tilde{\eta}\Vert\colon\Vert\tilde{F}\Vert\rightarrow\Vert\tilde{F^\prime}\Vert$.

On the other hand, we can see $\eta$ as a $(n+1)$-dimensional cube $\eta\colon\cube{n}\times I\rightarrow\lfburnside$ that restricts to $F$ on $\cube{n}\times\{1\}$ and to $F^\prime$ on $\cube{n}\times\{0\}$, and $\tilde{\eta}$ as a spatial refinement extending $\tilde{F}$ and $\tilde{F^\prime}$ given by Lemma \ref{lemma:recursiveb}. By the second part of Lemma \ref{lemma:propervogt}, $\Vert\tilde{\eta}\Vert$ is proper. By Remark \ref{remark_A1}, $\eta$ induces a well-defined map up to proper homotopy between the realizations $\vert F\vert$ and $\vert F^\prime\vert$.
\end{proof}

This completes the proof of Theorem \ref{thm:intro2}.

\subsection{The realization of a cube in the equivariant Burnside category}\label{ss:3.4}

Given a cube in the $G$-equivariant Burnside category $F\colon\cube{n}\rightarrow\gburnside{G}$, we can construct a spatial refinement $\tilde{F}_G$ of the quotient $F/G$, and then define an equivariant spatial refinement $\tilde{F}$ of $F$ by setting $\tilde{F}(u) = F(u)_+\wedge S^k$ and then defining the maps $\tilde{F}(u)\to \tilde{F}(v)$ as the only refinements that make the following diagram commute:
\[
    \xymatrix{
    \tilde{F}(u)\ar[d]\ar@{-->}[r] & \tilde{F}(v)\ar[d] \\
    \tilde{F}_G(u)\ar[r] &\tilde{F}_G(v)
    }
\]
Proceeding in the same manner with higher dimensional faces, we get a $G$-equivariant spatial refinement $\tilde{F}$ of $F$ that is free away from the basepoint. Now, following Section \ref{SectionRealizationBurnside}, the Vogt homotopy colimit $\|\tilde{F}\|$ of the homotopy coherent diagram $\tilde{F}_*$ is then a \GCW-complex and its associated spectrum $\vert F\vert = \Sigma^{-k}\Sigma^{\infty} \|\tilde{F}\|$ is a free \GCW-spectrum well-defined up to $G$-homotopy. Therefore we have:

\begin{proposition}\label{prop:quotient}
Let $F\colon \cube{n}\to \gburnside{G}$ be a cube. Then the realization $|F|$ is a locally compact \CW-spectrum with a free action of $G$ well-defined up to $G$-homotopy. 
Moreover, any natural transformation between cubes in the $G$-Burnside category induces a map between \CW-spectra well-defined up to $G$-homotopy.

Additionally, if $N\subset G$ is a normal subgroup and $F_N$ is the levelwise quotient of $F$ by $N$, then $|F_N|\simeq |F|/N$ equivariantly with respect to the action of $G/N$.
\end{proposition}

\begin{proof}
The spatial refinement $\tilde{F}_G$ constructed above is also a refinement of the composition $\cube{n}\to \gburnside{G}\to \lfburnside$, and therefore yields a locally compact spectrum.

If $\eta\colon F\rightarrow F^\prime$ is a natural transformation between cubes $F,F^\prime\colon \cube{n}\to \gburnside{G}$, it is a $(n+1)$-dimensional cube in $\gburnside{G}$ that restricts to $F$ on $\cube{n}\times\{1\}$ and $F^\prime$ on $\cube{n}\times\{0\}$. Let $\tilde{F}_G$ and $\tilde{F^\prime}_G$ be spatial refinements of the quotients $F/G$ and $F^\prime/G$. By Lemma \ref{lemma:recursiveb}, we can extend them to a spatial refinement $\tilde{\eta}_G$ of $\eta/G$, that is, a spatial refinement of $\eta$ as a cube in $\lfburnside$. We can lift $\tilde\eta_G$ to a spatial refinement of $\eta$ by the same arguments used in the discussion preceding this proposition. Then, $\eta$ induces a map $\tilde\eta\colon\tilde{F}\rightarrow\tilde{F^\prime}$ which in turn yields a map $\vert\eta\vert\colon\vert F\vert\rightarrow\vert F^\prime\vert$ well-defined up to $G$-homotopy.

For the last assertion, notice that the quotient of the $G$-equivariant spatial refinement $\tilde{F}_G$ by the subgroup $N$ is a $G/N$-equivariant spatial refinement of $F_N$. Since the action of $G$ on $\tilde{F}_G$ is free, taking quotient by $N$ commutes with the homotopy colimit used to build the realization of $F$, and therefore $|F_N|\simeq |F|/N$.
\end{proof}

\section{Duality for cubes in the Burnside category}\label{section:duality}

The Burnside category is equivalent to its opposite: reversing the role of the maps in a correspondence defines a functor $\dual\colon \smallburnside\to \smallburnside^{\op}$. This duality disappears in the locally finite Burnside category, but still exists in the bilocally finite Burnside category, since the locally finiteness assumption is imposed in both maps of each correspondence.

The same happens with the category $\bfinMod{R}$ of finite free $R$-modules with a basis: Taking duals defines a functor $\dual\colon \bfinMod{R}\to \bfinMod{R}^{\op}$ which is not the identity on objects, but the existence of a basis yields a canonical isomorphism between the finite free $R$-module $M$ and $\dual(M)$. The linearization functor $\linear\colon \smallburnside\to \bfinMod{R}$ commutes with the duality functors in the sense that there is a natural isomorphism between $\dual\circ \linear$ and $\linear\circ \dual$. 

The dual of a non-finite free $R$-module is not isomorphic to the original $R$-module. Nonetheless, this can be solved by changing the category: Define the category $\bMod{R}$ whose objects are modules endowed with a prescribed basis and morphisms $f\colon M\to N$ are homomorphisms such that for each basis element $v$ of $N$ there are finitely many basis elements $e_i$ of $M$ such that $(v^\vee)(f(e_i))\neq 0$. Define the \emph{finitely supported dual} of a module $M\in \bMod{R}$ as the subspace $\fdual(M)\subset \dual(M)$ of those functions $f\colon M\to R$ such that $f(e_i) = 0$ for all basis elements except for finitely many of them (alternatively, the subspace generated by the $e_i^\vee$).  Again, there is a canonical isomorphism between $\fdual(M)$ and $M$, and the linearization functor $\linear\colon \lfburnside \to \bMod{R}$ commutes with the finitely supported duality functor in the sense that there is a natural isomorphism between $\fdual\circ \linear$ and $\linear\circ \dual$. Additionally, the category $\bMod{R}$ has finite coproducts and the functor $\fdual$ commutes with them. Define $\bCh{R}$ as the category of chain complexes in $\bMod{R}$, and $\CWSpectra_{lc}$ as the category of \rlc \CW-spectra and proper maps.

Applying this discussion levelwise to cubes in the Burnside category, we obtain the following commutative diagram:
\[ \xymatrixcolsep{4pc}
	\xymatrix{
		\CWSpectra_{lc} \ar@/^1pc/[rrd]^{C_*} &&
        \\
		\lfburnside^{\cube{n}}\ar[u]^{|\cdot |} \ar[r]\ar@{<->}[d]_{\dual} \ar[r]^{\linear}& \Mod{R}_b^{\cube{n}} \ar@{<->}[d]^{\fdual} \ar[r]^{\Tot} & \bCh{R}\ar@{<->}[d]^{\fdual} 
        \\
		\lfburnside^{\cubeop{n}}\ar[r] & \Mod{R}_b^{\cubeop{n}} \ar[r] & \bCh{R}
	}
\]
The arrows in the lower left-hand square are compositions with the indicated functors, where we identify functors $\cubeop{n}\to \lfburnside^{\op}$ with functors $\cube{n}\to \lfburnside$, and the same with $\Mod{R}_b$. The bottom right map is defined so that the right square commutes.


The commutativity of the upper square follows by the same reasoning as in \cite{LLS20}. Let $\cD$ be the category $\cubeplus{n}$. Then $\Tot(F) \simeq \hocolim_{\cD} \linear\circ F$, and we have 
\[
	C_*|F| \simeq
	\Sigma^{-k}C_*\|\tilde{F}\| \simeq \hocolim_{\cD} \Sigma^{-k}C_*\tilde{F} \simeq \hocolim_{\cD} \linear\circ F\simeq \Tot(F). 
\]

Here is an application of this discussion to Khovanov homology that only uses the small Burnside category: Khovanov homology is the totalization of a contravariant cube $F$ of $R$-modules (middle bottom of the diagram). The work of Lawson, Lipshitz and Sarkar produces a (strictly unitary, lax) covariant cube $\cF$ in the Burnside category (left middle of the diagram) such that the cube of $R$-modules $\linear\circ\cF$ (center of the diagram) is dual to the cube of Khovanov. As a consequence, the dual of the cellular chain complex of $|F|$ is isomorphic to the Khovanov chain complex. 

When working equivariantly we have the following counterpart of the previous diagram:
\[
	\xymatrix{\xymatrixcolsep{10pc}
	G\text{-}\CWSpectra \ar@/^1pc/[rrrd]^{C_*} & & & 
        \\
		\gburnside{G}^{\cube{n}}\ar[u]^{|\cdot |} \ar[r] \ar@{<->}[d]_{\dual}   & \Mod{R[G]}^{\cube{n}} \ar@{<->}[d]^{\fdual} \ar[rr]^{\Tot} && \Ch(R[G])\ar@{<->}[d]^{\fdual} 
        \\
		\gburnside{G}^{\cubeop{n}}\ar[r] & \Mod{R[G]}^{\cubeop{n}} \ar[rr] && \Ch(R[G])
        }
\]

\section{Quantum annular homology and Burnside functor}\label{SectionQuantumAnnular}

In \cite{APS04} Asaeda, Przytycki and Sikora generalized Khovanov homology to the more general setting of links in {\it{thickened surfaces}}, that is, $[0,1]$-bundles over surfaces. In the particular case when the surface is the annulus $\mathbb{A} = S^1 \times [0,1]$, the homology is called {\it{annular Khovanov homology}}. The idea to adapt Viro’s approach in \cite{V02} to the setting of annular links (i.e., links embedded in the thickened annulus $\mathbb{A} \times [0,1]$) consists of considering diagrams as projections over the first factor $\mathbb{A}$ and noticing that after smoothing crossings to obtain Kauffman states, there are two different types of closed curves in $\mathbb{A}$: the ones bounding disks (trivial), and the ones parallel to the core of the annulus (essential). Authors manage to extend Viro’s differential to the case when essential curves are involved in the mergins and splittings. As in the classical case, annular Khovanov homology is functorial and its construction can be summarized in the annular Khovanov functor  $\mathcal{F}_\mathbb{A}\colon\cubeop{n}\rightarrow \Mod{\bZ}$. The totalization of the functor $\mathcal{F}_\mathbb{A}$ associated to a diagram $D$ is the annular Khovanov chain complex $CKh_{\mathbb{A}}(D)$.

 Some years later, Beliakova, Putyra and Wehrli  \cite{BPW19} introduced {\it{quantum annular Khovanov homology}}, a triply graded homology theory carrying an action which behaves well with the action of cobordisms. This construction can be summarized in a functor $\mathcal{F}_{\mathbb{A}_\mathfrak{q}}\colon \cubeop{n}\rightarrow \Mod{\Bbbk}$, a deformation of the annular Khovanov homology functor, where $\Bbbk=\mathbb{Z}[\mathfrak{q},\mathfrak{q}^{-1}]$ is a ring and $\mathfrak{q}\in\Bbbk$ a distinguished unit. As in the previous case, the totalization of the functor $\mathcal{F}_{\mathbb{A}_\mathfrak{q}}$  associated to a diagram $D$ is the quantum annular Khovanov chain complex $CKh_{\mathbb{A}_\mathfrak{q}}(D)$.

\subsection{The quantum annular Burnside functor}

In \cite{AKW21} the authors consider the finite cyclic group $G_r = \langle \mathfrak{q} \, | \, \mathfrak{q}^r = 1\rangle$ with $r\geq 1$, and construct the \emph{quantum annular Burnside functor} $F_\mathfrak{q}^r\colon\cube{n}\rightarrow\gburnside{G_r}$. The quotient $F^1_{\mathfrak{q}} = F_{\mathfrak{q}}^r/G_r$ is precisely the classical annular Khovanov Burnside functor from \cite[Section~4.3]{AKW21}, whose associated spectrum recovers the annular Khovanov homology of \cite{APS04}.

If one considers the infinite cyclic group $G = \langle \mathfrak{q}\rangle$ instead, their construction yields a $G$-equivariant cube $F_\mathfrak{q}\colon\cube{n}\rightarrow\hugeburnside$ in the big Burnside category, since the condition $\mathfrak{q}^r=1$ is not used in their argument. The quotient of this cube by the action of $G$ is $F_{\mathfrak{q}}^1$, which is a cube in the finite Burnside category. Therefore the cube $F_{\mathfrak{q}}$ defines a functor $F_{\mathfrak{q}}\colon \cube{n}\to \gburnside{G}$.

Denote the locally compact \GCW-spectrum $|F_{\mathfrak{q}}|$ obtained from a link diagram $D$ using Proposition \ref{prop:quotient} by $\mathcal{X}_{\mathbb{A}_\mathfrak{q}}(D)$. 

\begin{theorem}
The proper homotopy type of the locally compact \GCW-spectrum $\mathcal{X}_{\mathbb{A}_\mathfrak{q}}(D)$ is a link invariant. 
\end{theorem}

\begin{proof}
In \cite[Theorem~5.11]{AKW21} this theorem is proven for the functors $F_{\mathfrak{q}}^r$. Although they require the acting group to be finite, their argument applies in our setting, by Proposition \ref{prop:A-G-CW}.
\end{proof}

\begin{theorem}
If $D$ is a link diagram, there is an isomorphism of cochain complexes 
$$C_c^{\ast,cell}(\mathcal{X}_{\mathbb{A}_\mathfrak{q}}(D)) \cong \CKhannular(D).$$ In particular  $$H_c^\ast(\mathcal{X}_{\mathbb{A}_\mathfrak{q}}(D))\cong H_\ast(\CKhannular(D)).$$
\end{theorem}
\begin{proof}
 By the discussion in Section \ref{section:duality}, the cochain complex of $\|F_{\mathfrak{q}}\|$ with compact support is isomorphic to the dual of the totalization $\Tot(F_\mathfrak{q})$ (up to a shift), so it is isomorphic to the quantum annular Khovanov complex $CKh_{\mathbb{A}_\mathfrak{q}}(D)$.
\end{proof}

The following result is a consequence of the second part of Proposition \ref{prop:quotient}. 

\begin{corollary}
    The quotient of the spectrum $\mathcal{X}_{\mathbb{A}_\mathfrak{q}}(D)$ by the normal subgroup $r\bZ\subset \bZ$ yields the spectrum $\mathcal{X}^r_{\mathbb{A}_\mathfrak{q}}(D)$ of \cite{AKW21}.
\end{corollary}

After this corollary one can take the spectrum $\mathcal{X}_{\mathbb{A}_\mathfrak{q}}(D)/r\bZ$ as the definition of $\mathcal{X}^r_{\mathbb{A}_\mathfrak{q}}(D)$. This viewpoint avoids the need to treat the cases $r=1,2$ separately as in \cite[Section 4.3]{AKW21}.

\appendix
\section{Locally compact spectra}

In this appendix we introduce the category of locally compact \CW-spectra. From the definition, it will follow that there is a functor from this category to the category of locally finite chain complexes, and taking the cohomology of the finitely supported duals defines the \emph{compactly supported cohomology} of a locally compact \CW-spectrum. We then relate them to free \GCW-spectra and give a criterion for a map to be a $G$-homotopy equivalence.

\subsection{Relatively locally compact pointed CW-complexes} 
A continuous map $f\colon X\to Y$ between topological spaces is \emph{proper} if the preimage of each compact subset is compact. If $A\subset Y$ is a subspace, then $f$ is \emph{proper relative to $A$} if for every compact subspace $K\subset Y/A$, the quotient $f^{-1}(K)/f^{-1}(A\cap K)$ is compact. We will only consider the case when $A$ is a point.
\begin{definition}
 A \emph{relatively locally compact} \CW-complex (abbreviated as \rlc) is a pointed \CW-complex such that the union of the attaching maps of any given dimension is a proper map relative to the basepoint.
\end{definition} 

Equivalently, a pointed \CW-complex is \rlc if any open cell different from the basepoint intersects only finitely many cells of any given dimension. A \CW-complex is \emph{reduced} if it has a single $0$-cell.

\begin{example}
The reduced suspension of a locally compact pointed \CW-complex is a \rlc \CW-complex.
\end{example}

\begin{example}
The quotient $X/T$ of a $T$-complex by its tree $T$ \cite{Baues-Quintero, Muro} is a \rlc \CW-complex.
\end{example}

\begin{definition}
A cellular map between \rlc \CW-complexes is \emph{proper} if it is proper relative to the basepoint. 
\end{definition}

Equivalently, a cellular map $f$ is proper if the preimage of each open cell different from the basepoint intersects only finitely many cells of each dimension. 

\begin{example}
    The reduced suspension of a proper map between locally compact \CW-complexes is a proper map between \rlc \CW-complexes.
\end{example}

\begin{definition}
The \emph{compactly supported cochain complex} of a \rlc \CW-complex $X$ is the subcomplex $C^*_{c}(X,*)\subset C^*(X,*)$ of compactly supported cochains relative to the basepoint. The \emph{compactly supported cohomology} of $X$ is the cohomology $H^*(X,*)$ of this cochain complex.
\end{definition}

The relative chain complex $C_*(X,*)$ of a \rlc \CW-complex is a locally finite chain complex, and its finitely supported dual is precisely $C^*_{c}(X,*)\subset C^*(X,*)$. Proper maps induce homomorphisms between compactly supported cohomology groups. Properly homotopic proper maps induce the same homomorphism between compactly supported cohomology groups.

Let $G$ be a countable group. A free pointed \GCW-complex is a pointed \CW-complex $X$ with a cellular action of $G$ that fixes the basepoint and is free away from it. The quotient $X/G$ is a pointed \CW-complex, and we say that $X$ is $G$-compact if the quotient is compact. If that is the case, then $X$ is a \rlc \CW-complex. An equivariant map $f\colon X\to Y$ between $G$-finite free pointed \CW-complexes is always proper. An equivariant map $f\colon X\to Y$ is a \emph{$G$-homotopy equivalence} if there is an equivariant map $g\colon Y\to X$ such that $f\circ g$ and $g\circ f$ are equivariantly homotopic to the identity. In that case, $f$ is also a proper homotopy equivalence. The following proposition is a consequence of Theorem 5.3 in \cite{M71}. 

\begin{proposition}\label{prop:B-G-CW}
    Let $f\colon X\to Y$ be an equivariant map between finite dimensional free pointed \GCW-complexes. Then $f$ is a $G$-homotopy equivalence if and only if $f$ is a homotopy equivalence.
\end{proposition}

\subsection{Locally compact spectra}

Recall that a \CW-spectrum is a sequence of pointed \CW-complexes $\{E_n\}_{n\geq 0}$ together with structural maps $\varepsilon_n\colon \Sigma E_n\to E_{n+1}$ that include $\Sigma E_n$ as a subcomplex of $E_{n+1}$. 
\begin{definition}
    A finite dimensional \CW-spectrum $\bE$ is \emph{locally compact} if every \CW-complex $E_n$ is \rlc and the structural maps $\varepsilon_n$ are proper and eventually homeomorphisms.
\end{definition}

Proper functions and proper maps of finite dimensional locally compact \CW-spectra are defined as expected.

\begin{remark}\label{remark_A1}
    The suspension spectrum of a finite dimensional \rlc \CW-complex is locally finite. If $f\colon X\to Y$ is a proper map of finite dimensional locally finite \CW-complexes, then $\Sigma^{\infty} f$ is a proper function of finite dimensional locally finite \CW-spectra.
\end{remark}
The structural maps of a \CW-spectrum $\bE = \{E_n,\varepsilon_n\}$ induce chain homomorphisms $(\varepsilon_n)_\ast\colon C_\ast(E_n,*)\rightarrow \Sigma^{-1}C_\ast(E_{n+1},*)$. The \emph{chain complex of the \CW-spectrum $\bE = \{E_n\}_{n\geq 0}$} is the colimit
$$
    C_\ast(\bE)=\colim_{n} \Sigma^{-n}C_{\ast}(E_n,*).
$$
The chain complex of a locally compact \CW-spectrum is locally finite, so it has a well-defined finitely-supported dual, that will be called \emph{the compactly supported cochain complex of $\bE$.} The \emph{compactly supported cohomology $H^*_{c}(\bE;R)$ of a spectrum $\bE$} is the cohomology of the cochain complex $C^*_{c}(\bE)$.

A proper map of locally compact \CW-spectra induces a map between their compactly supported cochain complexes. A proper homotopy between two maps of locally compact \CW-spectra induces a homotopy between their induced maps on compactly supported cochain complexes.

A \emph{free \GCW-spectrum} is a \CW-spectrum $\bE = \{E_n,\varepsilon_n\}$ such that every \CW-complex $E_n$ is a free pointed \GCW-complex and the structure maps $\varepsilon_n$ are equivariant. The quotient $\bE/G$ is a \CW-spectrum, and we say that $\bE$ is \emph{$G$-finite} if the quotient is a finite \CW-spectrum. If that is the case and $\bE$ is finite-dimensional, then $\bE$ is a locally compact \CW-spectrum. An equivariant map between finite dimensional free $G$-finite \CW-spectra is always proper. An equivariant map $f\colon \bE\to \bE'$ is a \emph{$G$-homotopy equivalence} if there is an equivariant map $g\colon \bE'\to \bE$ such that $f\circ g$ and $g\circ f$ are equivariantly homotopic to the identity. The following is the stable version of Proposition \ref{prop:B-G-CW}:

\begin{proposition}\label{prop:A-G-CW}
    An equivariant map $f\colon \bE\to \bE'$ of free $G$-\CW-spectra is a $G$-homotopy equivalence if and only if it is a homotopy equivalence.
\end{proposition}

\bibliographystyle{abbrv}
\bibliography{aux/Bibliography.bib}

\end{document}